\documentclass[12pt]{article}
\usepackage{amsmath}
\usepackage{amsfonts}
\usepackage{amssymb}
\usepackage{amsthm}
\usepackage{fullpage}

\theoremstyle{plain}
\newtheorem{thm}{Theorem}[section]
\newtheorem{cor}[thm]{Corollary}
\newtheorem{lem}[thm]{Lemma}
\newtheorem{prop}[thm]{Proposition}
\theoremstyle{definition}
\newtheorem{defn}{Definition}[section]
\theoremstyle{remark}

\numberwithin{equation}{section}

\newcommand{\ex}{\text{e}}
\newcommand{\vr}{\boldsymbol{r}} 

\def\Pr{{\rm Pr}}

\begin{document}

\title{Bounds for the discrete correlation of infinite sequences on $k$ symbols and generalized Rudin-Shapiro sequences}
\author{E.~Grant, J.~Shallit, T.~Stoll}
\maketitle

\section{Introduction}

Pseudorandom sequences, i.e., deterministic sequences on finite alphabets with properties reminiscent of random sequences, are an intensively studied subject. We refer to the series of papers by Mauduit, S\'ark\"ozy and
coauthors~\cite{ACS08, CMS02, GMS04, MS97, MS98} among many others. A great part of the mentioned work deals with correlation measures for binary sequences and the problem to find large classes of finite pseudorandom binary sequences with small autocorrelation. Let $x=x_0x_1\cdots x_N\in\{-1,1\}^N$ be a finite word over the alphabet $\{-1,1\}$. Then the correlation measure of order $m$ of $x$ is defined as
\begin{equation}\label{corrfinite}
  U_m(x)=\max_{M,\vr} \left\vert \sum_{n=0}^M x_{n+r_1} x_{n+r_2}\cdots x_{n+r_m}\right\vert,
\end{equation}
where the maximum is taken over all $\vr=(r_1,r_2,\ldots,r_m)$ with $0\leq r_1<r_2<\cdots<r_m$ and $M$ such that $M+r_m\leq N$. In case of infinite words $x=x_0 x_1 \cdots$ the correlation of order $m$ is defined as
\begin{equation}\label{corrinfinite}
  V_m(x,M)=\sum_{n=0}^M x_{n+r_1} x_{n+r_2}\cdots x_{n+r_m},
\end{equation}
with fixed $\vr$. In contrast to $U_m(x)$, this definition does not take ``large-range correlations'' into account. In fact, $r_m$ could be $\Omega(N)$ for the finite word correlation~\cite{MS97}. Recently, Mauduit and S\'ark\"ozy~\cite{MS02} generalized several measures for pseudorandomness to finite sequences over $k$-letter alphabets. These distribution measures have been studied by B\'erczi~\cite{Be03} from a probabilistic point of view.

The aim of the present paper is to study the \textit{discrete correlation} among members of arbitrary infinite sequences over $k$ symbols, where we just take into account whether two symbols are identical. In the sequel, we denote by $\mathbb{N}$ the set of non-negative integers, and we assume that sums start with index $0$ (empty sums are supposed to be zero), unless otherwise stated. We further denote by $n \bmod k$ the unique integer $n'$ with $0\leq n'\leq k-1$ and $n\equiv n'$ (mod $k$). We use ``word'' and ``sequence'' interchangeably.

Let $x=x_0x_1\cdots$ be an infinite word over an alphabet of size $k$. Without loss of generality we may assume that $x_i \in \left\{0,1, \ldots, k-1\right\}$ for $i\in\mathbb{N}$. For vectors $(i_1,i_2,\ldots,i_m)$ with integers $i_j$ $(1\leq j\leq m)$ satisfying $0\leq i_1<i_2<\cdots<i_m$, define the \textit{discrete correlation coefficient $\delta(i_1,i_2,\ldots,i_m)$ of order $m$} by
$$
\delta(i_1,i_2,\ldots,i_m) = \left\{
\begin{array}{rl}
0, & \text{if } x_{i_1} = x_{i_2} = \cdots = x_{i_m} \text{;}\\
1, & \text{otherwise.}\\
\end{array} \right.
$$
Moreover, define $C_{\vr}$ for all fixed $\vr=(r_1,r_2,\ldots,r_m)$ with $0\leq r_1<r_2<\cdots<r_m$ by
\begin{equation}\label{cr}
  C_{\vr} = \liminf_{N \rightarrow \infty} \frac{1}{N}\sum_{n<N} \delta(n+r_1,n+r_2,\ldots,n+r_m).
\end{equation}

It is important to remark that for a random sequence (where every symbol is independently chosen with probability 1/k) the quantity $C_{\vr}$ equals $1-1/k^{m-1}$ with probability one. In this paper we investigate sequences with respect to this leading term. We first show by combinatorial means that for any infinite sequence on $k$ symbols the quantity $C_{\vr}$ cannot be too large for all $\vr$ (Theorem~\ref{thm1}). Our result, however, does not rule out the existence of deterministic sequences that actually attain our bound. We provide such a construction in the case of $m=2$ by introducing \textit{generalized Rudin-Shapiro sequences on $k$ symbols}, which extends a construction by Queff\'elec~\cite{Qu87} and H\o holdt, Jensen and Justesen~\cite{HJJ86, HJJ85}. The motivation stems from the fact that the autocorrelation $C_{(r_1,r_2)}$ of the infinite Rudin-Shapiro sequence on two symbols is small~\cite[Theorem~4]{MS98}. Our construction, however, gives a large class of sequences with small autocorrelation for any alphabet with cardinality $k$, whenever $k$ is prime or squarefree.

\medskip

The paper is structured as follows. In Section~\ref{sec1} we state the general bounds for the discrete correlation in Theorems~\ref{thm1} and~\ref{thm2}. In Section~\ref{sec2} we give the definition of generalized Rudin-Shapiro sequences. Sections~\ref{sec3} and~\ref{sec4} are devoted to the combinatorial proofs of Theorem~\ref{thm1} and~\ref{thm2}, respectively. In Section~\ref{sec4-LLL} we give the proof of Theorem~\ref{thm3} by using the Lov\'asz local lemma. Finally, in Sections~\ref{sec5} and~\ref{sec6} we give the proofs for Theorems~\ref{mtheo} and~\ref{mtheo2} by means of exponential sums.

\section{General bounds for the discrete correlation}\label{sec1}

We wish to establish upper bounds for $C_{\vr}$ as $\vr$ gets ``large''. To begin with, we normalize the vector $\vr$. For an integer sequence $T = (t_0, t_1, \ldots )$ with $t_i + r_1 \geq 0$ for $i\in\mathbb{N}$, we define
shifted versions of $C_{\vr}$, namely,
\begin{equation*}
C_{\vr,T} = \liminf_{N \rightarrow \infty} \frac{1}{N}\sum_{n<N} \delta(n+t_N+r_1,n+t_N+r_2,\ldots,n+t_N+r_m).
\end{equation*}

\begin{prop} \label{rprop}
Let $\vr=(r_1,r_2,\ldots,r_m)$ with $0\leq r_1<r_2<\cdots<r_m$, and let $T = (t_0, t_1, \ldots )$ be a sequence of integers with $t_i + r_1 \geq 0$ for all $i$.  If $t_N = o(N)$, then $C_{\vr,T} = C_{\vr}$.
\end{prop}
\begin{proof}
We note that
\begin{equation*}
C_{\vr,T} = \liminf_{N \rightarrow \infty} \frac{1}{N}\sum_{n=t_N}^{N+t_N-1} \delta(n+r_1,n+r_2,\ldots,n+r_m). \\
\end{equation*}
Since $\delta(n+r_1,n+r_2,\ldots,n+r_m) \in \{0,1\}$ for all $n$, the above sum differs from the corresponding sum in \eqref{cr} by at most $2 t_N$.  Thus if $t_N = o(N)$, then
\begin{equation*}
C_{\vr,T} = \liminf_{N \rightarrow \infty} \frac{1}{N}\left(\sum_{n<N} \delta(n+r_1,n+r_2,\ldots,n+r_m) + o(N)\right) = C_{\vr}. \qedhere
\end{equation*}
\end{proof}

By taking $T = (t,t,\ldots)$, Proposition~\ref{rprop} implies that $C_{\vr + t\mathbf{1}} = C_{\vr}$ for all constants $t \geq -r_1$.  We shall say $\vr$ is \emph{normalized} whenever $r_1 = 0$ and $r_1 < r_2 < \cdots < r_m$, and henceforth only consider normalized $\vr$.  In the $m=2$ case, we then have $\vr=(0,r_2)$ and we can establish an upper bound by taking the limit as $r_2$ approaches infinity.  We shall obtain the following result.

\begin{thm} \label{thm0} Let $x$ be an infinite word over an alphabet of size $k$.  Then
\begin{equation}\label{CRbound2}
  \liminf_{r_2 \rightarrow \infty}{C_{(0,r_2)}} \leq 1 - \frac{1}{k}.
\end{equation}
\end{thm}

In the next section we provide the construction of deterministic sequences with equality in~(\ref{CRbound2}).
More precisely, we show that for generalized Rudin-Shapiro sequences ($k$ prime or squarefree) we have
$$\inf_{r_2>0} \{C_{(0,r_2)}\}=1-\frac{1}{k}.$$
To generalize Theorem~\ref{thm0} to larger values of $m$, we must precisely define the notion of ``$\vr$ getting large''.  Let $||\cdot||$ be a norm on the finite dimensional vector space $\mathbb{R}^m$.  We will prove the following upper bound on $C_{\vr}$ as $||\vr||$ tends to infinity:

\begin{thm} \label{thm1}
 Let $x$ be an infinite word over an alphabet of size $k$.  Then for any $m \geq 2$ and any norm $||\cdot||$, we have
\begin{equation}\label{CRbound}
  \lim_{\lambda \rightarrow \infty} \left( \inf \left\{C_{\vr} : \vr \in \mathbb{N}^m, \; \vr \text{ normalized},\;  ||\vr|| \geq \lambda\right\} \right) \leq 1 - \frac{1}{k^{m-1}}.
\end{equation}
\end{thm}

We note that Theorem~\ref{thm0} is immediately implied by Theorem~\ref{thm1} by taking $m = 2$.  Theorem~\ref{thm1} is proven via a combinatorial argument in Section~\ref{sec3}.

In order to also consider the local autocorrelation properties of sequences, we define a related quantity. Again, let $x$ be an infinite word over an alphabet of size $k$.  For a given vector $\vr$ and positive integers $d$, we define
\begin{equation}\label{dr}
  D_{\vr}^d = \min_{n \geq 0} \left(\frac{1}{d}\sum_{i=n}^{n+d-1}\delta(i+r_1,i+r_2,\ldots,i+r_m)\right) \text{.}
\end{equation}
Note that for a random sequence on $k$ symbols, we necessarily have $D_{\vr}^d=0$ for all $\vr$ and $d$. We will prove that for a given vector $\vr$, the value of $C_{\vr}$ of an infinite sequence is an upper bound for all of the values of $D_{\vr}^d$:
\begin{thm} \label{thm2}
  Let $x$ be an infinite word over an alphabet of size $k$, $\vr$ be normalized and
  $d > 0$.  Then $D_{\vr}^d \leq C_{\vr}$.
\end{thm}

As an immediate consequence of Theorem~\ref{thm1} and Theorem~\ref{thm2}, we obtain an upper bound on $D_{\vr}^d$ as $||\vr||$ tends to infinity.
\begin{cor}\label{coro}
  Let $x$ be an infinite word over an alphabet of size $k$.  Then for any $m \geq 2$, $d > 0$, and norm $||\cdot||$, we have
\begin{equation}\label{DRbound}
  \lim_{\lambda \rightarrow \infty} \left( \inf \left\{D_{\vr}^d : \vr \in \mathbb{N}^m, \; \vr \text{ normalized}, \;||\vr|| \geq \lambda\right\} \right) \leq 1 - \frac{1}{k^{m-1}}.
\end{equation}
\end{cor}

An interesting example occurs when we choose a fixed $d>0$ and take $$\vr = (0,d,2d,\ldots,(m-1)d).$$ Then for each subword $w_1w_2\cdots w_m$ of $x$ with $\left|w_i\right| = d$ for all $i$, the number of indices $j$ where $|\{w_i\left[j\right] : 1 \leq i \leq m\}| > 1$ is at least $dD_{\vr}^d$. In this case, for sufficiently large $d$, we can get arbitrarily close to the bound in~(\ref{DRbound}).

\begin{thm}\label{thm3}
For all $\varepsilon>0$ there exist an infinite word $x$ over an alphabet of size $k$ and $d_0=d_0(\varepsilon)$ such that for all $d> d_0$ and ${\vr}=(0,d,2d,\ldots, (m-1)d)$ we have
$$D_{\vr}^d \geq  1-\frac{1}{k^{m-1}}-\varepsilon.$$
\end{thm}

\section{Generalized Rudin-Shapiro sequences}\label{sec2}

The quantity $C_{\vr}$ has been studied for various special sequences. A classical result of Mahler~\cite{Ma27} states that for the Thue-Morse sequences over $k$ symbols, the summatory correlation has no uniform leading term. On the contrary, Queff\'elec~\cite{Qu87} noted (referring to an unpublished result by Kamae) that the Rudin-Shapiro sequence indeed has the desired leading term, whenever $r$ is fixed. As for the hub of the present article, Mauduit and S\'ark\"ozy~\cite[Corollary after Theorem~4]{MS98} showed that for the correlation of order $2$ one may let $r_2=o(N)$ without losing this property. The following definition gives an extension to alphabets of size $k\geq 2$.

\begin{defn}\label{def}
Let $g: \; \{0,1,\ldots,k-1\}\times \mathbb{Z}\rightarrow \mathbb{Z}$, $(j,n)\mapsto g(j,n)$ be a function which is periodic in $n$ with period $k$. Furthermore, let $g$ be such that for all integers $u, i$ with $0\leq u< u+i\leq k-1$ we have
$$\{\;(g(u+i,n)-g(u,n)) \bmod k:\; 0\leq n\leq k-1\;\}=\{\;0,1,\ldots, k-1\;\}.$$
Then we call a sequence $(\hat{a}(n))_{n\geq 0}$ over the alphabet $\{0,1,\ldots,k-1\}$ a \textit{generalized Rudin-Shapiro sequence} if there exists a sequence of integers $(a(n))_{n\geq 0}$ such that $\hat{a}(n)\equiv a(n)$ (mod $k$) and
\begin{equation}\label{reccc}
  a(n k+j)=a(n)+g(j,n),\qquad 0\leq j\leq k-1, \quad n\geq 1.
\end{equation}
The function $g$ is called an \textit{admissible function}.
\end{defn}

\medskip

\textbf{Example 1:} A ``canonical'' admissible function $g$ in the sense of Definition~\ref{def} is
\begin{equation}\label{stand}
  g(j,n)=j\cdot (n \bmod k),
\end{equation}
which is Queff\'elec's generalization for the ordinary Rudin-Shapiro sequence~\cite[Section 4]{Qu87}. In this case $g(u+i,n)-g(u,n)\equiv in$ (mod $k$), and $\{in:\; 0\leq n\leq k-1\}$ runs for $i$ with $0\leq i\leq k-1$ through all residue classes mod $k$, provided $k$ is prime. In particular, for $k=2$ and
$$g(j,n)=\left\{
  \begin{array}{ll}
    1, & \hbox{if $j=1$, $n\equiv 1$ (mod $2$);} \\
    0, & \hbox{otherwise}
  \end{array}
\right.
$$
we get the \textit{Rudin-Shapiro sequence over the alphabet $\{0,1\}$}, namely, $$(\hat{a}(n))_{n\geq 0}=0,0,0,1,0,0,1,0,\ldots,$$ where the corresponding sequence $a(n)$ counts the number of subblocks $(1,1)$ in the binary expansion of $n$.

\medskip

\textbf{Example 2:} For $k=2$ and appropriate initial conditions, we get sequences which count \textit{any} fixed block of size two. For instance, by setting
$$g(1,0)=1, \qquad g(0,0)=g(1,1)=g(0,1)=0,$$
the resulting sequence $(\hat{a}(n))_{n\geq 0}$ counts (mod $2$) the number of subblocks $(01)$ in the binary expansion of $n$.

\medskip

\textbf{Example 3:}
For $k=3$ an admissible function other than~(\ref{stand}) is given by
$$g(j,n)=\left\{
  \begin{array}{ll}
    1, & \hbox{if $j\equiv n$ (mod $3$);} \\
    0, & \hbox{otherwise.}
  \end{array}
\right.$$
Here, the resulting sequence $(\hat{a}(n))_{n\geq 0}$ (with initial conditions $\hat{a}(0)=\hat{a}(1)=\hat{a}(2)=0$) gives the cumulative number of appearances (mod $3$) of subblocks $(00)$, $(11)$ and $(22)$ in the ternary expansion of integers.

\medskip

The following theorem shows that generalized Rudin-Shapiro sequences resemble the discrete autocorrelation behavior of random sequences if $m=2$.

\begin{thm}\label{mtheo}
Let
$$\hat{a}(0), \hat{a}(1), \hat{a}(2),\ldots$$ be a generalized Rudin-Shapiro sequence over $\{0,1,\ldots,k-1\}$ with $k$ prime. Moreover, let $0\leq r_1<r_2$. Then, as $N\rightarrow \infty$, we have
\begin{equation}
  \sum_{n<N} \delta(n+r_1,n+r_2) = \left(1-\frac{1}{k}\right)N+O_{k}\left((r_2-r_1)\log \frac{N}{r_2-r_1}+r_2\right),\label{propstat}
\end{equation}
where the implied constant only depends on $k$.
\end{thm}
In the proof, we give an explicit value for the implied constant. As an immediate consequence we note

\begin{cor}
  In the setting of Theorem~\ref{mtheo}, if $r_2=o(N)$ then
  $$\sum_{n<N} \delta(n+r_1,n+r_2)\sim \left(1-\frac{1}{k}\right)N.$$
\end{cor}

It seems natural to consider the cross product of two generalized Rudin-Shapiro sequences to prime bases to construct an extremal sequence for squarefree $k$. Let $k=p_1 p_2\cdots p_d$ be a product of pairwise distinct primes, and put $c_1=1$, $c_i=p_1p_2\cdots p_{i-1}$ for $2\leq i\leq d$. We define the sequence $(\hat{a}(n))_{n\geq 0}$ by
\begin{equation}\label{sqrfa1}
  \hat{a}(n)= a(n) \bmod k,
\end{equation}
where $(a(n))_{n\geq 0}$ is defined by
\begin{equation}\label{sqrfa2}
  a(n)=c_1 a_1(n)+c_2a_2(n)+\cdots +c_d a_d(n).
\end{equation}
Herein, $(a_i(n))_{n\geq 0}$ satisfies the recursive relation
\begin{equation}\label{sqrfa3}
  a_i(p_in+j)=a_i(n)+g_i(j,n),\qquad 1\leq i\leq d,
\end{equation}
for $n\geq 1$ and $0\leq j\leq p_i-1$. Again, the functions $g_i$ are admissible functions in the sense of Definition~\ref{def} for $1\leq i\leq d$.
Our next result gives an estimate for the correlation of order two.

\begin{thm}\label{mtheo2}
Let $k=p_1 p_2\cdots p_d$ with $d\geq 2$ be squarefree and denote by
$$\hat{a}(0), \hat{a}(1), \hat{a}(2),\ldots$$ a generalized Rudin-Shapiro sequence over $\{0,1,\ldots,k-1\}$
defined by~(\ref{sqrfa1}),~(\ref{sqrfa2}) and~(\ref{sqrfa3}). Moreover, let $0\leq r_1<r_2$ and $0<\gamma<1$.
Then, as $N\rightarrow \infty$, we have
\begin{align}
  \sum_{n<N} \delta(n+r_1,n+r_2)&= \left(1-\frac{1}{k}\right)N+O_{k}\Big((r_2-r_1) N^{1-\gamma/d}+(r_2-r_1) N^{1-\gamma} \log \frac{N^{\gamma/d}}{r_2-r_1}\nonumber\\
   &\quad\qquad\qquad\qquad\qquad\qquad +N^{\gamma}+r_1\Big),\label{propstat2}
\end{align}
where the implied constant only depends on $k$.
\end{thm}

\begin{cor}
  In the setting of Theorem~\ref{mtheo2}, if $r_2=o(N^{\gamma/d})$ then
  $$\sum_{n<N} \delta(n+r_1,n+r_2)\sim \left(1-\frac{1}{k}\right)N.$$
\end{cor}

\section{Proof of Theorem~\ref{thm1}}\label{sec3}
We need the following lemma for our proof of Theorem~\ref{thm1}.
\begin{lem} \label{lem1}
Suppose we have a multiset of $n$ distinct objects of $k$ types, and let $d \leq n$ be a fixed constant.  Then among the $\binom{n}{d}$ subsets of $d$ objects, the number containing at least one pair of objects of different types is at most
\[
\dfrac{n^d}{d!}\left(1-\dfrac{1}{k^{d-1}}\right).
\]
\end{lem}

\begin{proof}
Suppose we have $b_i$ objects of type $i$ for all $1 \leq i \leq k$.  Then we have $\binom{b_i}{d}$ subsets consisting entirely of objects of type $i$.  Thus the total number of subsets $P$ that contain at least one pair of objects of different types is
\begin{align*}
P &= \binom{n}{d} - \sum_{i=1}^{k}{\binom{b_i}{d}}\\
  &= \frac{1}{d!}\left(n(n-1)\cdots(n-d+1) - \sum_{i=1}^{k}{b_i(b_i-1)\cdots(b_i-d+1)}\right).
\end{align*}
Consider the polynomial $\phi(x) = x(x-1)\cdots(x-d+1) = e_1x + \cdots + e_dx^d$.  We rewrite our expression for $P$ in terms of $\phi$,
\begin{align*}
P &= \frac{1}{d!}\left(\phi(n) - \sum_{i=1}^{k}\phi(b_i)\right) \\
  &= \frac{1}{d!}\left(\phi(n) - \left(e_1\sum_{i=1}^{k}{b_i} + e_2\sum_{i=1}^{k}{b_i^2} + \cdots + e_d\sum_{i=1}^{k}{b_i^d}\right)\right).
\end{align*}
By the power means inequality,
\[
\frac{n}{k} = \frac{1}{k}\sum_{i=1}^{k}{b_i} \leq \left(\frac{1}{k}\sum_{i=1}^{k}{b_i^\nu}\right)^{1/\nu} \qquad \text{for all $\nu \geq 1$},
\]
and thus
\[
\left(\frac{n^\nu}{k^{\nu-1}}\right) \leq \sum_{i=1}^{k}{b_i^\nu}.
\]
We apply this bound to our expression for $P$ to yield the desired result,
\begin{align*}
P &\leq \frac{1}{d!}\left(\phi(n) - \left(e_1n + e_2\frac{n^2}{k} + \cdots + e_d\frac{n^d}{k^{d-1}}\right)\right) \\
  &= \frac{1}{d!}\left(n(n-1)\cdots(n-d+1) - k\cdot \frac{n}{k}\left(\frac{n}{k}-1\right)\cdots\left(\frac{n}{k}-d+1\right)\right) \\
  &\leq \frac{1}{d!}\left(n^d - k\left(\frac{n}{k}\right)^d\right) \\
  &= \frac{n^d}{d!}\left(1-\frac{1}{k^{d-1}}\right). \qedhere
\end{align*}
\end{proof}
With our lemma in hand, we now prove Theorem \ref{thm1}.  We proceed via contradiction.  Suppose that for some $m \geq 2$ and some norm $||\cdot||$ on $\mathbb{R}^m$, there exists an $\varepsilon > 0$ such that
$$
  \lim_{\lambda \rightarrow \infty} \left( \inf \left\{C_{\vr} : \vr \in \mathbb{N}^m,\; \vr \text{ normalized}, \;||\vr|| \geq \lambda\right\} \right) = 1 - \frac{1}{k^{m-1}} + \varepsilon.
$$
We assume without loss of generality that $\varepsilon < \dfrac{1}{k^{m-1}}$.  Our limit implies that there is some $\lambda_0 \in \mathbb{R}$ such that for all normalized $\vr \in \mathbb{N}^m$ with $||\vr|| \geq \lambda_0$ we have
\begin{equation} \label{Nlimit}
\liminf_{N \rightarrow \infty} \frac{1}{N}\sum_{i=0}^{N-1} \delta(i+r_1,i+r_2,\ldots,i+r_m) \geq 1 - \frac{1}{k^{m-1}} + \frac{\varepsilon}{2}.
\end{equation}
We define $\rho(\vr) = \max{\{r_j\}} - \min{\{r_j\}}$ to be the \emph{range} of $\vr$ and note that $\rho(\vr) = r_m$ whenever $\vr$ is normalized.  Let $\boldsymbol{r^*} = (0,\ldots,0,1) \in \mathbb{R}^m$ and let $p$ be an integer such that $p||\boldsymbol{r^*}|| \geq \lambda_0$.  Then whenever $\vr$ is normalized with $\rho(\vr) \geq p$, we have $||\vr|| \geq ||p\boldsymbol{r^*}|| = p||\boldsymbol{r^*}|| \geq \lambda_0$.  Hence, for all normalized $\vr$ with $\rho(\vr) \geq p$, we can pick $n_{\vr} \in \mathbb{N}$ by \eqref{Nlimit} such that for all $N \geq n_{\vr}$, we have
\begin{equation}\label{nrlimit}
  \frac{1}{N}\sum_{i=0}^{N-1}\delta(i+r_1,i+r_2,\ldots,i+r_m) \geq 1 - \frac{1}{k^{m-1}} + \frac{\varepsilon}{3}.
\end{equation}
To construct our counterexample, we ensure that we have selected $p$ such that
\begin{equation}\label{ppick}
p \geq m,
\end{equation}
and then pick $q \in \mathbb{N}$ such that the following both hold:
\begin{align}
(a)& \qquad q > \frac{18m^2(m-1)}{\varepsilon}\text{;} \label{qpick0} \\
(b)& \qquad q^{m-1} > \frac{9m(m-1)p^{m-1}}{\varepsilon}\text{.} \label{qpick1}
\end{align}
Since there are finitely many normalized $\vr \in \mathbb{N}^m$ with $p \leq \rho(\vr) \leq q$, we can then pick an $n \in \mathbb{N}$ such that the following both hold:
\begin{align}
(a)& \qquad n \geq n_{\vr} \text{ for all normalized $\vr$ with } p \leq \rho(\vr) \leq q \text{.} \label{npick0} \\
(b)& \qquad n > \frac{18qm!}{\varepsilon} \text{.} \label{npick1}
\end{align}

Now, for any set $U \subset \mathbb{N}$ with $|U| = m$, there is a unique normalized vector $\vr^U$ and integer offset $\mu(U)$ such that the vector $\vr^U + \mu(U)\mathbf{1}$ is an ordering of the elements of $U$.  We write $\delta(U)$ to denote the correlation coefficient associated to this vector, namely $\delta(U) = \delta(r_1^U + \mu(U), r_2^U + \mu(U), \ldots, r_m^U + \mu(U))$.  We also write $\rho(U) = \max{(U)} - \min{(U)}$ for the range of $U$.  It follows that $\rho(U) = \rho(\vr^U) = r_m^U$, and $\mu(U) = \min{(U)}$.  With these definitions in hand, we consider the following sum, which will be counted in two different ways to achieve our contradiction:
$$
S = \sum_{a=0}^{n-1}{\left(\sum_{\substack{U \subseteq \{a,\ldots,a+q-1\} \\ |U| = m}}{\delta(U)}\right)} \text{.}
$$

We first use Lemma~\ref{lem1} to bound $S$ from above.  The sum
$$ \sum_{\substack{U \subseteq \{a,\ldots,a+q-1\} \\ |U| = m}}{\delta(U)} $$
counts the number of subsets of $m$ elements from the multiset $\left[x_a, x_{a+1}, \ldots , x_{a+q-1}\right]$ that contain at least one pair of distinct symbols of the $k$ possible symbols.  Thus Lemma~\ref{lem1} applies, yielding
\begin{equation} \label{upperbound}
S \leq \sum_{a=0}^{n-1}{\frac{q^m}{m!}\left(1-\frac{1}{k^{m-1}}\right)} = \frac{nq^m}{m!}\left(1-\dfrac{1}{k^{m-1}}\right).
\end{equation}

Next, we will attempt to bound $S$ from below by expressing it in terms of partial sums of the form seen in \eqref{nrlimit}.  Our first goal will be to rearrange this sum according to the multiplicity of $\delta(U)$ for each $U$.  Sets $U$ will be subsets of $\{a,\ldots,a+q-1\}$ for more values of $a$ if they have lower range, so we sort the terms according to the value of $\rho(U)$, yielding
\begin{equation*}
S = \sum_{b=m-1}^{q-1}{\sum_{a=0}^{n-1}{\left(\sum_{\substack{U \subseteq \{a,\ldots,a+q-1\} \\ |U| = m \\ \rho(U) = b}}{\delta(U)}\right)}}.
\end{equation*}

For a given $U \subset \{0,\ldots,n+q-2\}$ with $|U| = m$, we have $U \subseteq \{a,\ldots,a+q-1\}$ if and only if $\min{(U)} \geq a$ and $\max{(U)} \leq a+q-1$.  Thus $U \subseteq \{a,\ldots,a+q-1\}$ for precisely those $a$ with $\mu(U) + \rho(U) - (q-1) \leq a \leq \mu(U)$.  However, when we rearrange our sum, we must count only those $a$ which also lie in the range $\{0,\ldots,n-1\}$.  We rewrite our sum as
\begin{equation*}
S = \sum_{b=m-1}^{q-1}{\left(\sum_{\substack{U \subseteq \{0,\ldots,n+q-2\} \\ |U| = m \\ \rho(U) = b}}{\left(\sum_{a=\max{\{\mu(U) + \rho(U) - (q-1),0\}}}^{\min{\{\mu(U), n-1\}}}{\delta(U)}\right)}\right)}.
\end{equation*}
We drop all terms containing elements less than $q$ or greater than $n-1$.  All the sets $U$ which remain will have $\mu(U) + \rho(U) - (q-1) \geq 0$ and $\mu(U) \leq n-1$, such that
\begin{align*}
S &\geq \sum_{b=m-1}^{q-1}{\left(\sum_{\substack{U \subseteq \{q,\ldots,n-1\} \\ |U| = m \\ \rho(U) = b}}{\left(\sum_{a=\mu(U) + \rho(U) - (q-1)}^{\mu(U)}{\delta(U)}\right)}\right)}\\
  &= \sum_{b=m-1}^{q-1}{\left(\sum_{\substack{U \subseteq \{q,\ldots,n-1\} \\ |U| = m \\ \rho(U) = b}}{\left((q-\rho(U))\delta(U)\right)}\right)}\\
  &= \sum_{b=m-1}^{q-1}{\left((q-b)\sum_{\substack{U \subseteq \{q,\ldots,n-1\} \\ |U| = m \\ \rho(U) = b}}{\delta(U)}\right)}.\\
\end{align*}
We now need to add back some of the terms we dropped and subtract away appropriate compensation.  We can choose $U \subseteq \{0,\ldots,n-1\}$ with $|U| = m$, $\rho(U)=b$ and $U \nsubseteq \{q,\ldots,n-1\}$ by picking $\min{(U)} \in \{0,\ldots,q-1\}$, taking $\max{(U)} = \min{(U)} + b$, and then choosing the remaining $m-2$ elements from $\{\min{(U)}+1, \ldots, \min{(U)} + b-1\}$.  There are $q\binom{b-1}{m-2}$ ways of doing this.  It is convenient to instead use $qb^{m-2}$ as an upper bound for this quantity; we then use the fact that $\delta(U) \in \{0,1\}$ to write
\begin{align*}
S &\geq \sum_{b=m-1}^{q-1}{\left(\left(q-b\right)\left(\left(\sum_{\substack{U \subseteq \{0,\ldots,n-1\} \\ |U| = m \\ \rho(U) = b}}{\delta(U)}\right) - qb^{m-2}\right)\right)}\\
  &> \sum_{b=m-1}^{q-1}{\left((q-b)\sum_{\substack{U \subseteq \{0,\ldots,n-1\} \\ |U| = m \\ \rho(U) = b}}{\delta(U)}\right)} - q^{m+1}.\\
\end{align*}
In a similar manner, we add back more terms so that we may consider all $U \subseteq \{0,\ldots,n+q-1\}$ with $|U| = m$ and $\rho(U)=b$, and subtract off another multiple of $q^{m+1}$ to compensate,
\begin{equation*}
S > \sum_{b=m-1}^{q-1}{\left((q-b)\sum_{\substack{U \subseteq \{0,\ldots,n+q-1\} \\ |U| = m \\ \rho(U) = b}}{\delta(U)}\right)} - 2q^{m+1}.
\end{equation*}
We now associate each set $U$ to its sorted vector $\vr^U + \mu(U)\mathbf{1}$ and group them according to their $\vr^U$ values.  Since we count each subset of $\{0,\ldots,n+q-1\}$ having range $\leq q-1$, we are certain to include $\vr + i\mathbf{1}$ for every normalized $\vr$ of range $\leq q-1$ and every offset $i$ from $0$ to $n$.  We drop any other terms and ignore those $\vr$ with $\rho(\vr) < p$ (recalling \eqref{ppick}, where we ensured that $p \geq m$), leaving us with
\begin{equation*}
S > \sum_{b=m-1}^{q-1}{\left((q-b)\sum_{\substack{\vr \in \mathbb{N}^m \\ \vr \text{ normalized} \\ \rho(\vr) = b}}{\sum_{i=0}^n{\delta(\vr+i\mathbf{1})}}\right)} - 2q^{m+1}.
\end{equation*}
Finally, we may use \eqref{nrlimit} to bound the inner sums from below, since for all $\vr$ with $\rho(\vr) \geq p$ we have $n \geq n_{\vr}$ by \eqref{npick0}.  We then simply count the number of normalized $\vr$ vectors of each range, obtaining
\begin{align}
S &> \sum_{b=p}^{q-1}{\left((q-b)\sum_{\substack{\vr \in \mathbb{N}^m \\ \vr \text{ normalized} \\ \rho(\vr) = b}}{n\left(1 - \frac{1}{k^{m-1}} + \frac{\varepsilon}{3}\right)}\right)} - 2q^{m+1}\nonumber\\
  &= n\left(1 - \frac{1}{k^{m-1}} + \frac{\varepsilon}{3}\right)\sum_{b=p}^{q-1}{\left((q-b)\binom{b-1}{m-2}\right)} - 2q^{m+1}\nonumber\\
  &\geq \frac{n}{(m-2)!}\left(1 - \frac{1}{k^{m-1}} + \frac{\varepsilon}{3}\right)\sum_{b=p}^{q-1}{\left((q-b)(b-m)^{m-2}\right)} - 2q^{m+1}.\label{final0}
\end{align}

We simplify and evaluate the remaining sum to get
\begin{align*}
\sum_{b=p}^{q-1}{\left((q-b)(b-m)^{m-2}\right)} &\geq \sum_{b=p}^{q-1}{\left((q+m-b)(b-m)^{m-2}\right)}-mq^{m-1}\\
                                                &\geq \sum_{b=p}^{q+m}{\left((q+m-b)(b-m)^{m-2}\right)}-2mq^{m-1}\\
                                                &= \sum_{b=p-m}^q{\left((q-b)b^{m-2}\right)}-2mq^{m-1}\\
                                                &\geq \sum_{b=0}^q{\left((q-b)b^{m-2}\right)}-2mq^{m-1}-qp^{m-1}\\
                                                &= q\sum_{b=0}^q{b^{m-2}}-\sum_{b=0}^{q-1}{b^{m-1}}-(2m+1)q^{m-1}-qp^{m-1}\\
                                                &\geq q\int_0^q{b^{m-2}\,\mathrm{d}b}-\int_0^q{b^{m-1}\,\mathrm{d}b}-2mq^{m-1}-qp^{m-1}\\
                                                &= \frac{q^m}{m(m-1)}-2mq^{m-1}-qp^{m-1}.
\end{align*}

We substitute this back into \eqref{final0} to obtain
\begin{equation} \label{finalS}
S > \frac{nq^m}{m!}\left(1 - \frac{1}{k^{m-1}} + \frac{\varepsilon}{3}\right) - 2q^{m+1} - \frac{2mnq^{m-1}}{(m-2)!} - \frac{nqp^{m-1}}{(m-2)!}.
\end{equation}

What remains is to eliminate the three leftover terms on the right hand side with the bounds we used when selecting $q$ and $n$.  First, by 
\begin{equation} \label{final1}
\left(\frac{nq^m}{m!}\right)\left(\frac{\varepsilon}{9}\right) > \frac{2mnq^{m-1}}{(m-2)!}.
\end{equation}

Second, by \eqref{qpick1}, we also picked $q$ such that
\begin{equation} \label{final2}
\left(\frac{nq^m}{m!}\right)\left(\frac{\varepsilon}{9}\right) > \frac{nqp^{m-1}}{(m-2)!}.
\end{equation}

Third, by \eqref{npick1}, we picked $n$ such that
\begin{equation} \label{final3}
\left(\frac{nq^m}{m!}\right)\left(\frac{\varepsilon}{9}\right) > 2q^{m+1}.
\end{equation}

Adding \eqref{final1}, \eqref{final2}, and \eqref{final3} together, we get
\begin{equation*}
\left(\frac{nq^m}{m!}\right)\left(\frac{\varepsilon}{3}\right) > 2q^{m+1} + \frac{2mnq^{m-1}}{(m-2)!} + \frac{nqp^{m-1}}{(m-2)!}
\end{equation*}
and we substitute this into \eqref{finalS} to obtain
\begin{equation*}
S > \frac{nq^m}{m!}\left(1 - \frac{1}{k^{m-1}}\right)
\end{equation*}
which contradicts \eqref{upperbound}, proving the desired result. \qed

\section{Proof of Theorem~\ref{thm2}}\label{sec4}
Suppose, for our sequence, that there exists some $m \geq 2$, $\vr \in \mathbb{N}^m$, and $d > 0$ such that $D_{\vr}^d > C_{\vr}$.  Let $\varepsilon = D_{\vr}^d - C_{\vr}$ and pick $p \in \mathbb{N}$ such that
\begin{equation*}
p > \frac{2dD_{\vr}^d}{\varepsilon} \text{.}
\end{equation*}
Then by our definition of $C_{\vr}$, there is some $n \geq p$ such that
\begin{equation*}
\frac{1}{n}\sum_{i=0}^{n-1}\delta(i+r_1,\ldots,i+r_m) < C_{\vr} + \frac{\varepsilon}{2} \text{.}
\end{equation*}
Dividing $n$ by $d$, we let $n = ad + b$, where $a$ and $b$ are non-negative integers and $b < d$.  Then rearranging our expression and applying the definition of $D_{\vr}^d$ yields:
\begin{align*}
C_{\vr} &> \frac{1}{n}\sum_{i=0}^{n-1}\delta(i+r_1,\ldots,i+r_m) - \frac{\varepsilon}{2}\\
        &= \frac{1}{n}\left(\sum_{i=0}^{a-1}\sum_{j=id}^{id+d-1}\delta(j+r_1,\ldots,j+r_m) + \sum_{i=ad}^{ad+b-1}\delta(j+r_1,\ldots,j+r_m)\right) - \frac{\varepsilon}{2}\\
        &\geq \frac{1}{n}\left(\sum_{i=0}^{a-1}\left(dD_{\vr}^d\right) + \sum_{i=ar}^{ar+b-1}\delta(j+r_1,\ldots,j+r_m)\right) - \frac{\varepsilon}{2}\\
        &\geq \frac{adD_{\vr}^d}{n} - \frac{\varepsilon}{2}\\
        &\geq D_{\vr}^d - \frac{dD_{\vr}^d}{n} - \frac{\varepsilon}{2} \text{.}
\end{align*}
However, since
\begin{equation*}
n \geq p > \frac{2dD_{\vr}^d}{\varepsilon},
\end{equation*}
we then have
\begin{equation*}
\frac{dD_{\vr}^d}{n} < \frac{\varepsilon}{2},
\end{equation*}
and substituting this into the above yields
\begin{equation*}
C_{\vr} > D_{\vr}^d - \frac{\varepsilon}{2} - \frac{\varepsilon}{2}= D_{\vr}^d - \varepsilon= C_{\vr} \text{.}
\end{equation*}
Thus we have a contradiction, and so we have $D_{\vr}^d \leq C_{\vr}$ for all $\vr$ and $d$. \qed

\section{Proof of Theorem~\ref{thm3}}\label{sec4-LLL}

It is sufficient to show that for all integers $k, m \geq 2$ and all real numbers $\varepsilon > 0$,
there exist an integer $d_0$ and an infinite word $x = x_0 x_1 x_2 \cdots$ over a $k$-letter alphabet such
that for every integer $d > d_0$ and $i \geq 0$ there are
at least $(1 - {1 \over k^{m-1}} - \varepsilon)$ positions where
the $m$ words
$$x_i \cdots x_{i+d-1}, \ x_{i+d} \cdots x_{i+2d-1},\ \ldots, \ x_{i+(m-1)d} \cdots x_{i+md - 1}$$
do not all agree. We use the Lov\'asz local lemma to show the existence of finite words
of every sufficiently long length satisfying the condition.  The existence
of an infinite word then follows from the usual compactness argument.

Here is the statement of the Lov\'asz local lemma, as taken
from \cite[Chap.~5]{AS00}.

\begin{lem}
Let $A_1, A_2, \ldots, A_T$ be events in a probability space, with a dependency digraph $D = (S, E)$.
Suppose there exist real numbers $u_1, u_2, \ldots, u_T$ with
$0 \leq u_i < 1$ for $1 \leq i \leq T$ such that
\begin{equation}
 \Pr(A_i) \leq u_i \prod_{(i,j) \in E} (1-u_j)
\label{lll}
\end{equation}
for $1 \leq i \leq T$.
Then the probability that none of the events $A_1, A_2, \ldots, A_T$
occur is $\geq \prod_{1 \leq i \leq T} (1 - u_i)$.
\end{lem}

Let $A_{i,d}$ denote the event that there are $<t$ positions where
the $m$ words
$$x_i \cdots x_{i+d-1}, \ x_{i+d} \cdots x_{i+2d-1},\ \ldots, \ x_{i+(m-1)d} \cdots x_{i+md - 1} $$
do not all agree. Moreover, let $S$ be the space of all such events $A_{i,d}$ and $(S,E)$ the dependency digraph
specifying when one event is dependent on another, which corresponds to overlapping ranges of the word being constructed.

To evaluate $\Pr[A_{i,d}]$ it suffices to count the number of such strings.  First, we choose the values for the
symbols of the first string, $x_i, \ldots, x_{i+d-1}$, which can be done in $k^d$ ways.  Next, we
choose the precise number of positions $j$ in which the $m$ strings will
fail to agree, and the positions themselves.  This can be done
in $\sum_{0 \leq j < t} {d \choose j}$ ways.
For each such position, there are $k^{m-1} -1$ ways to choose the
symbols of the remaining $m-1$ strings in such a way that they do not
universally agree with the first string.  The remaining symbols in the
last $m-1$ strings are now completely determined, as they must agree with
the symbols in the corresponding position in the first string.
The total number of such strings is therefore
$$P = k^d \sum_{0 \leq j < t} {d \choose j} (k^{m-1} - 1)^j.$$
We therefore find
$$\Pr[A_{i,d}] =  {P \over {k^{md}}}
= \sum_{0 \leq j < t} {d \choose j} \left( {{k^{m-1}-1} \over {k^{m-1}}}
\right)^j \left( {1 \over {k^{m-1}}} \right)^{d-j}.$$

To estimate this sum we use the following classical estimate on
the tail of the binomial distribution, which is a version of
Hoeffding's inequality \cite{Ho63}:
\begin{lem}
Suppose $0 < p < 1$, and let $t, d$ be positive integers
with $t \leq dp$.  Then
$$\sum_{0 \leq j \leq t} {d \choose j} p^j (1-p)^{d-j} \leq
e^{-2(dp-t)^2/d}.$$
\label{hoeffding}
\end{lem}
If we now take $t = (1-{1 \over k^{m-1}} - \varepsilon) d$,
$p = {{k^{m-1} - 1} \over {k^{m-1}}}$, we obtain
$$\Pr[A_{i,d}] \leq e^{-2d \varepsilon^2}.$$
Now fix $n$, the length of the string.  We want none of the events
$A_{j,s}$ for $d_0 \leq s \leq n/m$, $0 \leq j \leq n-ms$, to take place.
Choose $u_{j,s} = e^{-{1 \over 2} s \varepsilon^2}$.
Then
\begin{eqnarray*}
\prod_{((i,d),(j,s)) \in E} (1-u_{j,s})
&=& \prod_{ {{i-ms+1 \leq j \leq i+md-1} \atop {0 \leq j \leq n-ms}} \atop {d_0 \leq s \leq n/m }} (1-u_{j,s})\\
&\geq & \prod_{s \geq d_0} (1-u_{j,s})^{md+ms-1}.
\end{eqnarray*}
Taking logarithms, we get
$$
\sum_{((i,d),(j,s)) \in E} \log (1-u_{j,s})
\geq
\sum_{s \geq d_0} (md+ms-1) \log (1-u_{j,s}).
$$

Provided $u_{j,s}$ is sufficiently small, we can bound $\log(1-u_{j,s})$ with $-c u_{j,s}$ for some constant
$c$. Hence we get
\begin{align*}
\sum_{s \geq d_0} (md+ms-1) &\log (1-u_{j,s})\\
& \geq \sum_{s \geq d_0} - (md+ms-1) c e^{-{1 \over 2} \varepsilon^2 s} \\
&= - (md-1) c \sum_{s \geq d_0} e^{-{1 \over 2} \varepsilon^2 s}
- mc \sum_{s \geq d_0} s e^{-{1 \over 2}  \varepsilon^2 s} \\
&= -(md-1) c\, {{e^{-{1 \over 2} \varepsilon^2(d_0-1)}} \over
{e^{{1 \over 2} \varepsilon^2} - 1}} -
mc  \,{{e^{-{1 \over 2} \varepsilon^2(d_0-1)} (1-d_0) +
d_0 e^{-{1 \over 2} \varepsilon^2(d_0-2)}} \over
	{ (e^{{1 \over 2} \varepsilon^2} -1)^2}}.
\end{align*}

Now choose $d_0$ large enough so that
$${{e^{-{1 \over 2} \varepsilon^2(d_0-1)}} \over
{e^{{1 \over 2} \varepsilon^2} - 1}} \leq  {{\varepsilon^2} \over {2mc}},$$
and also large enough so that
$$
{{e^{-{1 \over 2} \varepsilon^2(d_0-1)} (1-d_0) +
d_0 e^{-{1 \over 2} \varepsilon^2(d_0-2)}} \over
	{ (e^{{1 \over 2} \varepsilon^2} -1)^2}} \leq {{\varepsilon^2 d_0 } \over {2mc}}.$$
It follows that

\begin{eqnarray*}
\log \left( u_{i,d} \prod_{ ((i,d),(j,s)) \in E  } (1-u_{j,s})
\right) &\geq &
- {1 \over 2} \varepsilon^2 d -(md-1) c  {{\varepsilon^2} \over {2mc}}
- mc {{\varepsilon^2 d_0 } \over {2mc}} \\
& \geq &
- {1 \over 2} \varepsilon^2 d - {1 \over 2} \varepsilon^2 d - {1 \over 2} \varepsilon^2 d_0\\
& \geq & -{3 \over 2} \varepsilon^2 d   \\
& \geq & -2 \varepsilon^2 d  \\
& \geq & \log \Pr[A_{i,d}],
\end{eqnarray*}
as desired. Hence, by the Lov\'asz local lemma, it follows that the probability that none
of the events $A_{j,s}$ occur is $\geq \prod_{((i,d),(j,s))\in E}(1-u_{j,s})>0$,
and hence such a string of length $n$ exists. \qed

\section{Proof of Theorem~\ref{mtheo}}\label{sec5}

Before turning to the proof of Theorem~\ref{mtheo}, we need one auxiliary tool.
We rewrite the left-hand-side expression of~(\ref{propstat}) in terms of exponential sums.
As usual, set $\ex(z)=e^{2\pi\mathrm{i}z}$ for $z\in\mathbb{R}$.

\begin{prop}\label{use}
For any infinite word $x_0 x_1 x_2\cdots$ over $\{0,1,\ldots,k-1\}$ we have
$$  \sum_{n<N}\delta(n+r_1,n+r_2)=
  N\left(1-\frac{1}{k}\right)-\frac{1}{k}\sum_{1\leq h<k}\; \sum_{n<N} \ex\left(\frac{h}{k}(x_{n+r_2}-x_{n+r_{1}})\right).
$$
\end{prop}

\begin{proof}
  The proof is based on the relation
  \begin{equation}\label{exp}
  \sum_{0\leq h<k} \ex\left(\frac{hu}{k}\right)=\left\{
                                                    \begin{array}{ll}
                                                      0, & \hbox{if $k\nmid u$;} \\
                                                      k, & \hbox{if $k \mid u$.}
                                                    \end{array}
                                                  \right.
  \end{equation}
  First, since $x_n\in\{0,1,2,\ldots,k-1\}$ we notice that $k\mid (x_{n+r_2}-x_{n+r_1})$ if and only if
  $x_{n+r_2}=x_{n+r_1}$. Therefore,
  \begin{align*}
    \sum_{n<N}\delta(n+r_1,n+r_2)&=N-\sum_{n<N} \; \frac{1}{k}  \sum_{0\leq h<k}\ex\left(\frac{h}{k}\left(x_{n+r_2}-x_{n+r_1}\right)\right)\\
     &\quad=N\left(1-\frac{1}{k}\right)-\sum_{n<N}\; \frac{1}{k} \sum_{1\leq h<k}\ex\left(\frac{h}{k}\left(x_{n+r_2}-x_{n+r_{1}}\right)\right).\qedhere
  \end{align*}
\end{proof}
In view of Theorem~\ref{mtheo} and Proposition~\ref{rprop} it suffices to show that
for all $1\leq h\leq k-1$ we have
\begin{equation}\label{condition}
 \sum_{n<N} \ex\left(\frac{h}{k}(\hat{a}(n+r)-\hat{a}(n))\right)=
  O_k\left(r \log\left(\frac{N}{r}\right)+r\right),
\end{equation}
where the implied constant only depends on $k$. Since $\ex(z+1)=\ex(z)$, the left-hand-side sum in~(\ref{condition}) can be rewritten in the form
\begin{equation}\label{condition2}
  \gamma_N(r)=\sum_{n<N} \ex\left(\frac{h}{k}(a(n+r)-a(n))\right).
\end{equation}
In the sequel we will need the generalized quantities
\begin{equation}\label{gammhat}
  \gamma_N(r,f)=\sum_{n<N} \ex\left(\frac{h}{k}(a(n+r)-a(n))\right)\ex\left(\frac{hf(n)}{k}\right),
\end{equation}
where $f:\; \mathbb{N}\rightarrow \mathbb{Z}$ is an arbitrary periodic function with period $k$.
We first show that for all such $f$ we have $\gamma_N(1,f)=O_k(\log N)$ for $N>k$. We will then use induction on $r$ to prove~(\ref{condition}), which in turn proves Theorem~\ref{mtheo}.

We follow the reasoning of Mauduit~\cite{M01}. Regarding~(\ref{gammhat}) we split the summation over $n<N$ up according to the residue class of $n$ modulo $k$. We obtain
\begin{align*}
  \gamma_{kN+j}(1,f)&=\sum_{n<kN+j}\ex\left(\frac{h}{k}(a(n+1)-a(n))\right)\ex\left(\frac{hf(n)}{k}\right)\\
 &= \sum_{i=0}^{k-1}\sum_{kn+i<kN+j}\ex\left(\frac{h}{k}(a(kn+i+1)-a(kn+i))\right)\ex\left(\frac{hf(i)}{k}\right).
\end{align*}
Thus,
\begin{align}
  \gamma_{kN+j}(1,f)&=\quad\sum_{n=0}^{k-1}\ex\left(\frac{h}{k}(a(n+1)-a(n))\right)\ex\left(\frac{hf(n)}{k}\right)\label{sum1}\\
 &\qquad +\sum_{u=0}^{j-1}\ex\left(\frac{h}{k}(a(kN+u+1)-a(kN+u))\right)\ex\left(\frac{hf(u)}{k}\right)\label{sum2}\\
 &\qquad +\sum_{u=0}^{k-2} \ex\left(\frac{hf(u)}{k}\right)\sum_{1\leq n<N}\ex\left(\frac{h}{k}(a(kn+u+1)-a(kn+u))\right)\label{sum3}\\
 &\qquad +\ex\left(\frac{hf(k-1)}{k}\right)\sum_{1\leq n<N}\ex\left(\frac{h}{k}(a(kn+k)-a(kn+k-1))\right).\label{sum4}
\end{align}
The sums~(\ref{sum1}) and~(\ref{sum2}) are trivially bounded by $k+j\leq 2k-1$. Concerning~(\ref{sum3}) we note that for $0\leq u\leq k-2$ we have
\begin{align*}
  \sum_{1\leq n<N} &\ex\left(\frac{h}{k}\left(a(kn+u+1)-a(kn+u)\right)\right)\\
  &=\sum_{1\leq n<N} \ex\left(\frac{h}{k}\left(a(n)+g(u+1,n)-a(n)-g(u,n)\right)\right)\\
  &=\sum_{1\leq n<N} \ex\left(\frac{h}{k}(g(u+1,n)-g(u,n))\right).
\end{align*}
By our assumption $g(u+1,n)-g(u,n)$ runs through a complete residue system mod $k$ for $1\leq n\leq k$, so this sum is bounded in modulus by $k/2$. Therefore,~(\ref{sum3}) is bounded by $k(k-1)/2$. Finally, we rewrite the sum in~(\ref{sum4}) in the form
\begin{align*}
  \sum_{1\leq n<N} &\ex\left(\frac{h}{k}\left(a(kn+k)-a(kn+k-1)\right)\right)\\
  &=\sum_{1\leq n<N} \ex\left(\frac{h}{k}\left(a(n+1)+g(0,n+1)-a(n)-g(k-1,n)\right)\right)\\
  &=\sum_{1\leq n<N} \ex\left(\frac{h}{k}\left(a(n+1)-a(n)\right)\right) \ex\left(\frac{h\hat{f}(n)}{k}\right),
\end{align*}
where $\hat{f}(n)=g(0,n+1)-g(k-1,n)$ is again periodic with period $k$ in $n$. Summing up, we get
\begin{equation}\label{gr1}
  \vert \gamma_{kN+j}(1,f)\vert \leq \vert \gamma_{N}(1,\hat{f}) \vert+\frac{k}{2}\,(k+3).
\end{equation}
From~(\ref{gr1}) and $\vert \gamma_{n}(1,f) \vert\leq k-1$ for $1\leq n\leq k-1$ and all $f$ we get by induction that for all $k$-periodic functions $f$ and all $N>k$,
\begin{equation}\label{gr2}
  \vert \gamma_N(1,f)\vert \leq \frac{k(k+3)}{2\log k}\; \log N +k-1.
\end{equation}
For our induction on $r$ to work, we need one more initial value, namely
$$\gamma_N(0,f)=\sum_{n<N} \ex\left(\frac{hf(n)}{k}\right)$$ which satisfies
\begin{equation}\label{crux}
  \vert \gamma_N(0,f) \vert \leq \frac{k}{2}, \quad \mbox{if }f(\{0,1,\ldots,k-1\})=\{0,1,\ldots,k-1\}.
\end{equation}

Now, let us consider the general case with $r=kM+i>0$ where $M\geq 0$ and $0\leq i\leq k-1$ but $(M,i)\neq (0,0)$. Similarly to~(\ref{sum1})--(\ref{sum4}) we have
\begin{align}
  \gamma&_{kN+j}(kM+i,f)=\nonumber\\
&\quad \sum_{u=0}^{k-2}\ex\left(\frac{hf(u)}{k}\right)\sum_{1\leq n<N} \ex\left(\frac{h}{k}(a(kn+u+kM+i)-a(kn+u))\right)\label{sum1a}\\
  &+\ex\left(\frac{hf(k-1)}{k}\right)\sum_{1\leq n<N} \ex\left(\frac{h}{k}(a(kn+k-1+kM+i)-a(kn+k-1))\right)\label{sum2a}\\
  &+O(1)\nonumber,
\end{align}
where the implied constant is bounded in modulus by $2k-1$.
We again need a close inspection of the two infinite sums~(\ref{sum1a}) and~(\ref{sum2a}). First, suppose $i\neq 0$. We
rewrite the sum~(\ref{sum1a}) in the form
\begin{align}
&\qquad\sum_{u=0}^{k-1-i} \ex\left(\frac{hf(u)}{k}\right)\sum_{1\leq n<N} \ex\Big(\frac{h}{k}(a(n+M)+g(u+i,n+M)
\nonumber\\
&\quad\qquad\qquad\qquad\qquad\qquad\qquad-a(n)-g(u,n) )\Big)\nonumber\\
&\quad +
\sum_{u=k-i}^{k-2} \ex\left(\frac{hf(u)}{k}\right)\sum_{1\leq n<N} \ex\Big(\frac{h}{k}(a(n+M+1)+g(u+i-k, n+M+1)\nonumber\\
&\quad\qquad\qquad\qquad\qquad\qquad\qquad-a(n)-g(u,n) )\Big)\nonumber\\
&=\quad
\sum_{u=0}^{k-1-i} \ex\left(\frac{hf(u)}{k}\right)\sum_{1\leq n<N} \ex\left(\frac{h}{k}\left(a(n+M)-a(n)\right)\right)\ex\left(\frac{hf_1(n)}{k}\right)\nonumber\\
&\quad +
\sum_{u=k-i}^{k-2} \ex\left(\frac{hf(u)}{k}\right)\sum_{1\leq n<N} \ex\left(\frac{h}{k}\left(a(n+M+1)-a(n)\right)\right)\ex\left(\frac{hf_2(n)}{k}\right),\nonumber
\end{align}
where
\begin{align*}
  f_1(n)&=g(u+i,n+M)-g(u,n), & \mbox{for } 0\leq u\leq k-1-i,\\
  f_2(n)&=g(u+i-k,n+M+1)-g(u,n), & \mbox{for } k-i\leq u\leq k-2.
\end{align*}
Using~(\ref{gammhat}) this yields
\begin{align}
  &\sum_{u=0}^{k-2}\ex\left(\frac{hf(u)}{k}\right)\sum_{1\leq n<N} \ex\left(\frac{h}{k}(a(kn+u+kM+i)-a(kn+u))\right)
 \label{ssz}\\
&=\sum_{u=0}^{k-1-i} \ex\left(\frac{hf(u)}{k}\right)\gamma_N(M,f_1)
+\sum_{u=k-i}^{k-2} \ex\left(\frac{hf(u)}{k}\right)\gamma_N(M+1,f_2)+O(1),\nonumber
\end{align}
where the $O(1)$-term comes from including $n=0$ into~(\ref{ssz}) and therefore is trivially bounded in modulus by $(k-i)+(i-1)=k-1$. Consider the second sum~(\ref{sum2a}) and let $i\neq 0$. Then
\begin{align*}
  &a(k(n+M+1)+i-1)-a(kn+k-1)\\
&\qquad \qquad = a(n+M+1)-a(n)+g(i-1,n+M+1)-g(k-1,n).
\end{align*}
Therefore,
\begin{align}
  \Big\vert \;\ex\left(\frac{hf(k-1)}{k}\right)&\sum_{1\leq n<N} \ex\left(\frac{h}{k}(a(kn+k-1+kM+i)-a(kn+k-1))\right)\;\Big\vert\nonumber\\
  &\leq \quad \left \vert\gamma_N(M+1,f_3)\right\vert +1\label{inot0},
\end{align}
where $f_3(n)=g(i-1,n+M+1)-g(k-1,n)$. Now, from~(\ref{sum1a}),~(\ref{sum2a}),~(\ref{ssz}) and~(\ref{inot0}) we see that
\begin{align}
  \vert \gamma_{kN+j}(kM+i,f)\vert &\leq \vert \gamma_N(M,f_1)\vert\cdot (k-i)+\vert \gamma_N(M+1,f_2)\vert\cdot (i-1)\nonumber\\
&\quad +\vert \gamma_N(M+1,f_3)\vert+1+(2k-1)+(k-1).\label{final1d}
\end{align}
Plugging in $M=0$, using~(\ref{gr2}) and~(\ref{crux}) and observing that $f_1(n)=g(u+i,n)-g(u,n)$ permutes $\{0,1,\ldots,k-1\}$ by assumption, we get
$$\vert \gamma_{kN+j}(i,f) \vert \leq \frac{k(k-1)(k+3)}{2\log k}\; \log N+\frac{k}{2}\,(2k+3),\qquad 1\leq i\leq k-1.$$
This implies that for $1\leq i\leq k-1$ and all functions $f$ with period $k$ we have
\begin{align}
  \vert \gamma_N(i,f)\vert &\leq \frac{k(k-1)(k+3)}{2\log k}\;\log\left(\frac{N}{k}\right)
+\frac{k}{2}\,(2k+3),\qquad N>k.\label{case}
\end{align}
On the other hand, if $0\leq u\leq k-1$ then
$$a(k(n+M)+u)-a(kn+u)= a(n+M)-a(n)+g(u,n+M)-g(u,n),$$
so by joining~(\ref{sum1a}) and~(\ref{sum2a}) in case that $i= 0$ we directly get
\begin{equation}\label{final2a}
  \vert \gamma_{kN+j}(kM,f)\vert \leq \sum_{u=0}^{k-1}\left(
  \vert \gamma_N(M,f_4)\vert+1\right)+(2k-1),
\end{equation}
where $f_4(n)=g(u,n+M)-g(u,n)$. Therefore, by~(\ref{gr2}) and~(\ref{final2a}) applied for $M=1$ we get
\begin{equation}\label{case2}
  \vert \gamma_N(k,f)\vert \leq \frac{k^2(k+3)}{2\log k}\;\log \left(\frac{N}{k}\right)+k^2+2k-1,
\end{equation}
provided $N>k$. Therefore, for all $N>k$,
\begin{equation}\label{case3}
  \vert \gamma_N(i,f)\vert \leq \frac{k^2(k+3)}{2\log k}\;\log \left(\frac{N}{k}\right)+k^2+2k-1,
\end{equation}
for the whole range $1\leq i\leq k$. We now start our induction on the parameter $r=kM+i$. We iterate~(\ref{final1d})
and~(\ref{final2a}) with~(\ref{case3}) as an initial value to obtain for $r=k^s+1, k^s+2,\ldots, k^{s+1}$
with $s\geq 0$ and for all $N>k^{s+1}$,
\begin{align*}
  \vert \gamma_{N}(r,f)\vert &\leq \frac{k^2(k+3)}{2\log k}\;k^s\; \log\left(\frac{N}{k^{s+1}}\right)
+k^s(k^2+2k-1)+\sum_{j=0}^{s-1} (3k-1)k^j\nonumber\\
  &\leq \frac{k^2(k+3)}{2\log k}\;k^s\;\log\left(\frac{N}{k^{s+1}}\right) + \frac{k^s(k^3+k^2)}{k-1}.
\end{align*}
This finishes the proof of Theorem~\ref{mtheo}.\qed

\section{Proof of Theorem~\ref{mtheo2}}\label{sec6}

For the proof of Theorem~\ref{mtheo2} it suffices to show that for all $1\leq h\leq k-1$
and $0<\gamma<1$ we have
\begin{equation}
  \sum_{n<N}\, \ex\left(\frac{h}{k}\left(a(n+r)-a(n)\right)\right)
\ll N^{\gamma}+ r N^{1-\gamma/d} +r N^{1-\gamma}\log\left(\frac{N^{\gamma/d}}{r}\right),\label{toest}
\end{equation}
where the implied constant only depends on $k$. We follow Kim~\cite[Section~4]{Ki99}, however suitably modifying the argument
to deal with the function $a$ not being $k$-additive in the usual sense. We need some more notation.
Let $b=(b_1,b_2,\ldots,b_d)$ and set $$P_b=\{n\in \mathbb{N}:\quad n\equiv b_i \mbox{ mod } p_i^{s_i}, \quad 1\leq i\leq d\},$$
where $s_i$ is the unique integer with $p_i^{s_i}\leq N^{\gamma/d}<p_i^{s_i+1}$. Since the  $p_i$'s denote different primes by assumption, we have
$$\#\{n\in\mathbb{N}: \quad n\in P_b\}=\frac{N}{\prod_{i=1}^d p_i^{s_i}}+O(1).$$
Further set
\begin{align*}
  \mathcal{B} &=\{(b_1,b_2,\ldots,b_d): \qquad 0\leq b_i<p_i^{s_i} \quad \mbox{ for } 1\leq i\leq d\},\\
  \mathcal{B}_0 &=\{(b_1,b_2,\ldots,b_d): \qquad 0\leq b_i<p_i^{s_i}-r \quad \mbox{ for } 1\leq i\leq d\}.
\end{align*}
Now, consider $n=n_i p_i^{s_i}+b_i$ where $0\leq b_i< p_i^{s_i}-r$. We may assume that $n_i\geq 1$, which is true
for most $n$, i.e. $N^{\gamma/d}\leq n< N$ (the error term of $N^{\gamma/d}$ is negligible in the final estimate).  Write
\begin{align*}
  b_i+r &= \beta'_{s_i-1} p_i^{s_i-1}+\beta'_{s_i-2} p_i^{s_i-2}+\cdots+\beta'_0,\\
  b_i &= \beta_{s_i-1} p_i^{s_i-1}+\beta_{s_i-2} p_i^{s_i-2}+\cdots+\beta_0
\end{align*}
where $\beta_\nu, \beta'_\nu\in\{0,1,\ldots,p_i-1\}$ for $0\leq \nu<s_i$. Furthermore, set
\begin{align*}
  v_i&=\max(j:\quad \beta_j'\neq 0, \quad 0\leq j\leq s_i-1),\\
  w_i&=\max(j:\quad \beta_j\neq 0, \quad 0\leq j\leq s_i-1),
\end{align*}
which indicate the uppermost nonzero coefficients in the expansions.
Then by~(\ref{sqrfa3}) we can rewrite $a_i(n+r)-a_i(n)$ in the form
\begin{align*}
    & a_i\left(n_i p_i^{s_i}+\beta'_{s_i-1} p_i^{s_i-1}+\cdots + \beta'_0\right)- a_i\left(n_i p_i^{s_i}+\beta_{s_i-1} p_i^{s_i-1}+\cdots + \beta_0\right)\\
   &\quad= a_i(n_i)+g_i(\beta'_{s_i-1},n_i)+\sum_{\nu=0}^{s_i-2} g_i(\beta'_\nu,\beta'_{\nu+1})\\
   &\quad\qquad - \left( a_i(n_i)+g_i(\beta_{s_i-1},n_i)+\sum_{\nu=0}^{s_i-2} g_i(\beta_\nu,\beta_{\nu+1})\right)\\
   &\quad= g_i(\beta'_{s_i-1},n_i)-g_i(\beta_{s_i-1},n_i)+\sum_{\nu=0}^{s_i-2} \left(g_i(\beta'_\nu,\beta'_{\nu+1})-g_i(\beta_\nu,\beta_{\nu+1})\right)\\
   &\quad=a_i(b_i+r)-a_i(b_i)+\mu_i(b_i,r,n_i),
\end{align*}
where
\begin{align*}
\mu_i(b_i,r,n_i)&=
g_i(\beta'_{s_i-1},n_i)-g_i(\beta_{s_i-1},n_i)+\sum_{\nu=v_i}^{s_i-2} g_i(\beta'_\nu,\beta'_{\nu+1})
   -\sum_{\nu=w_i}^{s_i-2}g_i(\beta_\nu,\beta_{\nu+1}).
\end{align*}
Consequently,
\begin{align}
  \sum_{n<N}\, \ex\left(\frac{h}{k}\left(a(n+r)-a(n)\right)\right)
& = \quad \sum_{n<N} \prod_{i=1}^d \ex\left(\frac{h}{k}\,c_i\left(a_i(n+r)-a_i(n)\right)\right)\nonumber\\
& = \quad \sum_{b\in \mathcal{B}_0}\sum_{\substack{n<N\\ n\in P_b}} \prod_{i=1}^d \ex\left(\frac{h}{k}\,c_i\left(a_i(b_i+r)-a_i(b_i)+\mu_i(b_i,r,n_i)\right)\right)
\nonumber\\
&\qquad\quad + \sum_{b\in \mathcal{B}\setminus \mathcal{B}_0}\sum_{\substack{n<N\\ n\in P_b}} \ex\left(\frac{h}{k}\,\left(a(n+r)-a(n)\right)\right),\nonumber
\end{align}
which equals
\begin{align}
&\quad\sum_{b\in \mathcal{B}} \prod_{i=1}^d \ex\left(\frac{h}{k}\,c_i\left(a_i(b_i+r)-a_i(b_i)\right)\right) \sum_{\substack{n<N\\ n\in P_b}}
\prod_{i=1}^d \ex\left(\frac{h}{k}\,c_i \mu_i(b_i,r,n_i)\right)\label{splitsum1}\\
&\qquad\quad + \sum_{b\in \mathcal{B}\setminus \mathcal{B}_0}\sum_{\substack{n<N\\ n\in P_b}} \left(\ex\left(\frac{h}{k}\left(a(n+r)-a(n)\right)\right)\right.\nonumber\\
&\qquad\qquad\qquad\qquad\qquad\quad\left.-\prod_{i=1}^d \ex\left(\frac{h}{k}\,c_i\left(a_i(b_i+r)-a_i(b_i)+\mu_i(b_i,r,n_i)\right)\right)\right).\label{splitsum2}
\end{align}
The second sum~(\ref{splitsum2}) is trivially bounded by (we follow~\cite{Ki99})
\begin{align}
  2\; \vert \mathcal{B}\setminus \mathcal{B}_0\vert \cdot \#\{n<N: \: n\in P_b\}&
   \ll \left(\sum_{i=1}^d \frac{r}{p_i^{s_i}}\prod_{j=1}^d p_j^{s_j}\right)\left(\frac{N}{\prod_{i=1}^d p_i^{s_i}}+O(1)\right)\nonumber\\
  &\ll rN^{1-\gamma/d},\label{zwei}
\end{align}
which is one of the error terms in the estimate. Now, consider the first sum~(\ref{splitsum1}). Let
$$\mathcal{B}^r=\{b\in \mathcal{B}: \quad v_i=w_i \mbox{ and } \beta_{v_i}=\beta_{w_i}' \mbox{ for all } 1\leq i\leq d\}.$$
Obviously, for every $b\in \mathcal{B}^r$ we have $\mu_i(b_i,r,n_i)=0$ for all $n<N$, $n\in P_b$.
We use a similar splitting as above, such that~(\ref{splitsum1}) satisfies
\begin{align*}
&\ll \quad \sum_{b\in \mathcal{B}} \prod_{i=1}^d \ex\left(\frac{h}{k}\,c_i\left(a_i(b_i+r)-a_i(b_i)\right)\right) \sum_{\substack{n<N\\ n\in P_b}}
1\\
&\qquad +2\,\vert \mathcal{B}\setminus \mathcal{B}^r\vert \left(\frac{N}{\prod_{i=1}^d p_i^{s_i}}+O(1)\right).
\end{align*}
Our next task is to establish a bound for $\vert \mathcal{B}\setminus \mathcal{B}^r\vert$. Let $p_i^{t_i}\leq r<p_i^{t_i+1}$. We have to count the number of $b_i$'s
with $0\leq b_i<p_i^{s_i}$ such that performing the addition $b_i+r$ gives rise to a carry propagation which is transported to the digits $\beta_{v_i}$ of $b_i$,
thus giving a contribution to $\mu_i(b_i,r,n_i)$. A necessary condition for this effect is that
$$\beta_{t_i+1}=\beta_{t_i+2}=\cdots=\beta_{s_i-2}=p_i-1.$$
Hence
\begin{align*}
  \vert \mathcal{B}\setminus \mathcal{B}^r\vert &\leq \sum_{i=1}^d \left(p_i^{t_i+1}+(s_i-1-t_i)p_i^{t_i+2}\right)\\
  &\ll\sum_{i=1}^d\left(r+p_ir\left(\frac{\log N^{\gamma/d}}{\log p_i}-\log r\right)\right)\\
  &\ll r \log N^{\gamma/d}.
\end{align*}
Summing up, we obtain
\begin{align*}
  &\sum_{n<N}\, \ex\left(\frac{h}{k}\left(a(n+r)-a(n)\right)\right)=\\
   &\qquad\sum_{b\in\mathcal{B}}\prod_{i=1}^d \ex\left(\frac{h}{k}\,c_i\left(a_i(b_i+r)-a_i(b_i)\right)\right)\sum_{\substack{n<N\\ n\in P_b}} 1
  +O\left(rN^{1-\gamma/d}+rN^{1-\gamma} \log N^{\gamma/d}\right)\\
  &= \quad \prod_{i=1}^d \sum_{b_i=0}^{p_i^{s_i}-1}\ex\left(\frac{h}{k}\,c_i\left(a_i(b_i+r)-a_i(b_i)\right)\right)\left(\frac{N}{\prod_{i=1}^{d}p_i^{s_i}}+O(1)\right)
  +O\left(rN^{1-\gamma/d}\right)\\
  &= \quad N \prod_{i=1}^d \frac{1}{p_i^{s_i}}\sum_{b_i=0}^{p_i^{s_i}-1}\ex\left(\frac{h}{k}\,c_i\left(a_i(b_i+r)-a_i(b_i)\right)\right)
  +O\left(N^\gamma+rN^{1-\gamma/d}\right).
\end{align*}
Finally, we show how to obtain the saving in the exponent, which again finishes the proof of Theorem~\ref{mtheo2}.
Since $c_i=p_1 p_2\cdots p_{i-1}$, we see that for every $h$ there exists
an index $l$ with $1\leq l\leq d$ and
$$\frac{h}{k}\;c_l=\frac{h p_1 p_2\cdots p_{i-1}}{p_1p_2\cdots p_d}=\frac{h'}{p_l},$$
with $\gcd(h',p_l)=1$. Applying Theorem~\ref{mtheo} with $k=p_l$ and estimating the other factors trivially, we get
$$\sum_{n<N}\, \ex\left(\frac{h}{k}\left(a(n+r)-a(n)\right)\right)
  \ll N^{1-\gamma} r\log \frac{N^{\gamma/d}}{r}+N^{1-\gamma} r+N^\gamma+r N^{1-\gamma/d},
$$
which gives the statement of the theorem. \qed

{\small

{\footnotesize
E.~Grant ({\tt egrant@uwaterloo.ca}), J.~Shallit ({\tt shallit@cs.uwaterloo.ca}),
T.~Stoll\\ ({\tt tstoll@cs.uwaterloo.ca}): Faculty of Mathematics, School of Computer Science, University of Waterloo, Waterloo, ON, Canada.}

\end{document}